\documentclass[12pt]{amsart}

\usepackage{amsfonts}
\usepackage{amsmath}
\usepackage{amssymb}
\usepackage{latexsym}
\usepackage{amscd}
\usepackage{amsthm}
\usepackage{subfigure}
\usepackage{graphicx}
\usepackage[all]{xy}
\usepackage{color}

\newtheorem{thm}{Theorem}[section]
\newtheorem{prop}[thm]{Proposition}
\newtheorem{lem}[thm]{Lemma}
\newtheorem{cor}[thm]{Corollary}
\theoremstyle{definition}

\newcommand{\Z}{\mathbb{Z}}
\newcommand{\M}{\mathcal{M}}
\newcommand{\I}{\mathcal{I}}
\newcommand{\C}{\mathcal{C}}
\newcommand{\Pu}{\mathcal{P}}
\newcommand{\F}{\mathcal{F}}

\author[R. Kobayashi]{Ryoma Kobayashi}
\address{
Department of General Education,\endgraf
Ishikawa National College of Technology,\endgraf
Tsubata, Ishikawa, 929-0392, Japan
}
\email{kobayashi\_ryoma@ishikawa-nct.ac.jp}

\begin{document}

\title[A normal generating set for $\I(N_g^b)$]
{A normal generating set for the Torelli group of a compact non-orientable surface}

\maketitle

\begin{abstract}
For a compact surface $S$, let $\I(S)$ denote the \textit{Torelli group} of $S$.
For a compact orientable surface $\Sigma$, $\I(\Sigma)$ is generated by two types of mapping classes, called \textit{bounding simple closed curve maps} (BSCC maps) and \textit{bounding pair maps} (BP maps) (see \cite{Po} and \cite{Pu}).
For a non-orientable closed surface $N$, $\I(N)$ is generated by BSCC maps and BP maps (see \cite{HK}).
In this paper, we give an explicit normal generating set for $\I(N_g^b)$, where $N_g^b$ is a genus-$g$ compact non-orientable surface with $b$ boundary components for $g\geq4$ and $b\geq1$.
\end{abstract}

\section{Introduction}

For $g\geq1$ and $b\geq0$, let $N_g^b$ denote a genus-$g$ compact connected non-orientable surface with $b$ boundary components, and let $N_g=N_g^0$.
In this paper, we regard $N_g^b$ as a surface obtained by attaching $g$ M\"obius bands to a sphere with $g+b$ boundary components, as shown in Figure~\ref{non-ori}.
We call each of these M\"obius bands attached to this sphere a \textit{cross cap}.
The \textit{mapping class group} $\M(N_g^b)$ of $N_g^b$ is the group consisting of isotopy classes of all diffeomorphisms over $N_g^b$ which fix each point of the boundary.
The \textit{Torelli group} $\I(N_g^b)$ of $N_g^b$ is the subgroup of $\M(N_g^b)$ consisting of elements acting trivially on the integral first homology group $H_1(N_g^b;\Z)$ of $N_g^b$.
The Torelli group of a compact orientable surface is generated by two types of mapping classes, called \textit{bounding simple closed curve maps} (BSCC maps) and \textit{bounding pair maps} (BP maps) (see \cite{Po} and \cite{Pu}).
In particular, Johnson \cite{J} showed that the Torelli group of an orientable closed surface is finitely generated by BP maps.
Hirose and the author \cite{HK} showed that $\I(N_g)$ is generated by BSCC maps and BP maps for $g\geq4$.
In this paper, we give an explicit normal generating set for $\I(N_g^b)$ consisting of BSCC maps and BP maps, for $g\geq4$ and $b\geq1$.

\begin{figure}[htbp]
\includegraphics[scale=0.5]{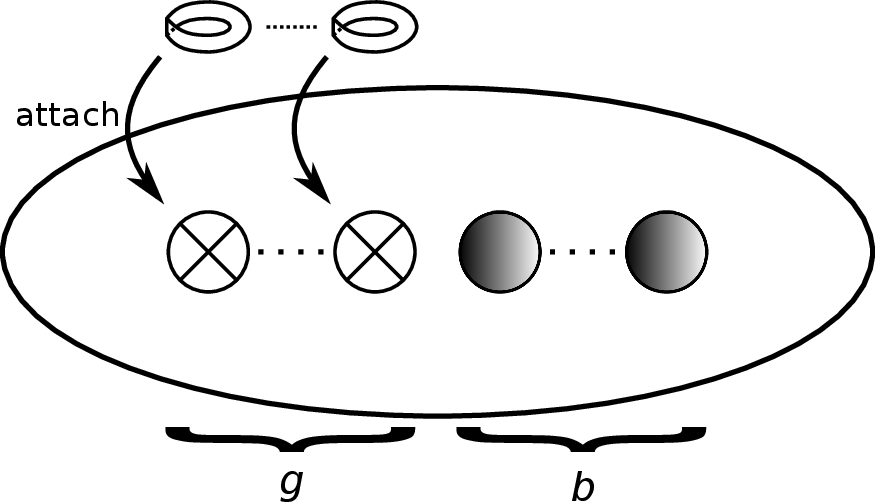}
\caption{A genus-$g$ compact connected non-orientable surface with $b$ boundary components.}\label{non-ori}
\end{figure}

We remark that there are other definitions of the Torelli group for surfaces with boundary (see \cite{Pu}).

Let $N$ be a compact connected non-orientable surface.
For a simple closed curve $c$ on $N$, we call $c$ an \textit{$A$-circle} (resp. an \textit{$M$-circle}) if its regular neighborhood is an annulus (resp. a M\"obius band), as shown in Figure~\ref{a-m}.
For an $A$-circle $c$, we can define the mapping class $t_c$, called the \textit{Dehn twist} about $c$, and the direction of the twist is indicated by a small arrow written beside $c$ as shown in Figure~\ref{dehn}.
We can notice that an $A$-circle (resp. an $M$-circle) passes through the cross caps even (resp. odd) number of times.

\begin{figure}[htbp]
\subfigure[A-circles.]{\includegraphics[scale=0.5]{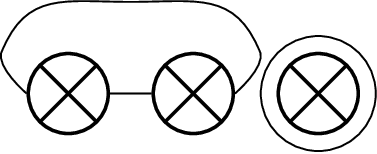}}\hspace{1cm}
\subfigure[M-circles.]{\includegraphics[scale=0.5]{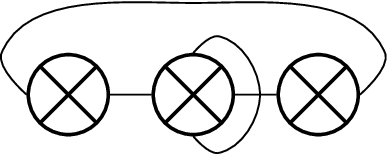}}
\caption{}\label{a-m}
\end{figure}

\begin{figure}[htbp]
\includegraphics[scale=0.5]{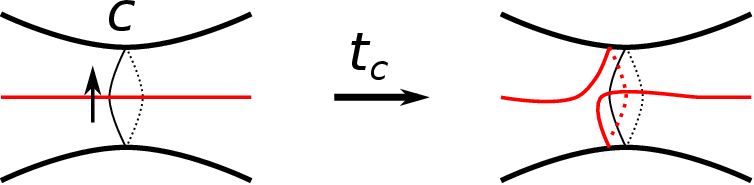}
\caption{The Dehn twist $t_c$ about $c$.}\label{dehn}
\end{figure}

Let $\alpha$, $\beta$, $\beta^\prime$, $\gamma$, $\delta_i$, $\rho_i$, $\sigma_{ij}$ and $\bar{\sigma}_{ij}$ be simple closed curves on $N_g^b$ as shown in Figure~\ref{loops}.
The main result of this paper is as follows.

\begin{thm}\label{main-thm}
Let $g\geq5$ and $b\geq0$.
In $\M(N_g^b)$, $\I(N_g^b)$ is normally generated by $t_\alpha$, $t_\beta{}t_{\beta^\prime}^{-1}$, $t_{\delta_i}$, $t_{\rho_i}$, $t_{\sigma_{ij}}$ and $t_{\bar{\sigma}_{ij}}$ for $1\leq{i,j}\leq{b-1}$ with $i<j$.
In $\M(N_4^b)$, $\I(N_4^b)$ is normally generated by $t_\alpha$, $t_\beta{}t_{\beta^\prime}^{-1}$, $t_{\delta_i}$, $t_{\rho_i}$, $t_{\sigma_{ij}}$, $t_{\bar{\sigma}_{ij}}$ and $t_\gamma$ for $1\leq{i,j}\leq{b-1}$ with $i<j$.
\end{thm}

\begin{figure}[htbp]
\subfigure[Loops $\alpha$, $\beta$, $\beta^\prime$, $\delta_i$ and $\bar{\sigma}_{ij}$.]{\includegraphics[scale=0.5]{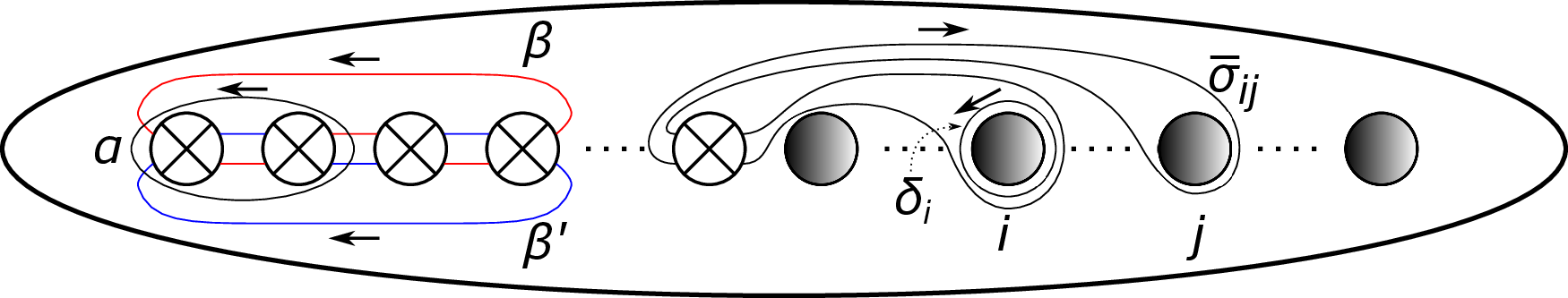}}
\subfigure[Loops $\gamma$, $\rho_i$ and $\sigma_{ij}$.]{\includegraphics[scale=0.5]{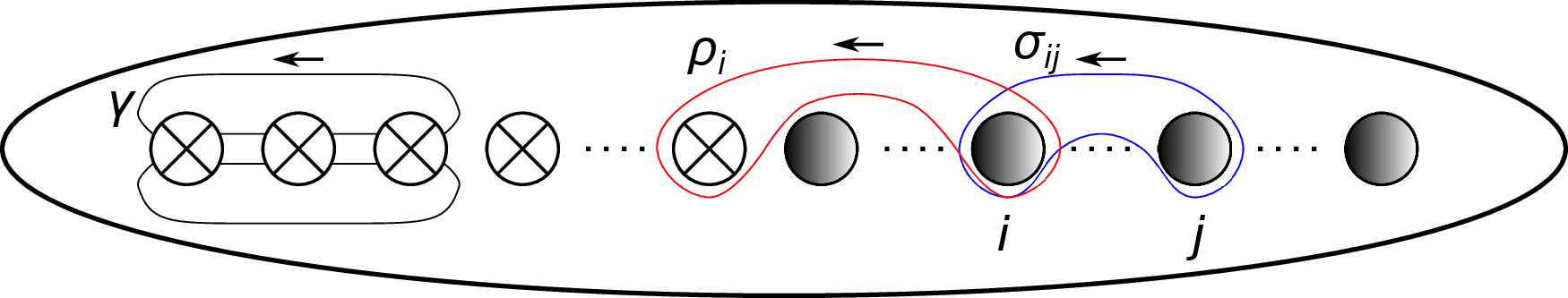}}
\caption{}\label{loops}
\end{figure}

In this paper, for $f,g\in\M(N_g^b)$, the composition $gf$ means that we first apply $f$ and then $g$.

\section{Basics on mapping class groups for non-orientable surfaces}

\subsection{On mapping class groups for non-orientable surfaces}\

Mapping class groups for orientable surfaces are generated by only Dehn twists.
Lickorish showed that $\M(N_g)$ is generated by Dehn twists and \textit{$Y$-homeomorphisms} and that the subgroup of $\M(N_g)$ generated by all Dehn twists is an index $2$ subgroup of $\M(N_g)$ (see \cite{L1,L2}).
Hence $\M(N_g)$ is not generated by only Dehn twists.
On the other hand, since a $Y$-homeomorphism acts trivially on $H_1(N_g;\Z/{2\Z})$, $\M(N_g)$ is not generated by only $Y$-homeomorphisms.

Chillingworth \cite{C} found a finite generating set for $\M(N_g)$.
$\M(N_1)$ and $\M(N_1^1)$ are trivial (see \cite{E}).
Finite presentations for $\M(N_2)$, $\M(N_2^1)$, $\M(N_3)$ and $\M(N_4)$ are obtained by \cite{L1}, \cite{St1}, \cite{BC} and \cite{Sz}, respectively.
Paris-Szepietowski \cite{PS} obtained a finite presentation of $\M(N_g^b)$  for $b=0,1$ and $g+b>3$.

Let $N$ be a non-orientable surface, and let $a$ and $m$ be an oriented $A$-circle and an $M$-circle on $N$ respectively such that $a$ and $m$ intersect each other transversely at only one point.
We now define a $Y$-homeomorphism $Y_{m,a}$.
Let $K$ be a regular neighborhood of $a\cup{b}$ in $N$, and let $M$ be a regular neighborhood of $m$ in the interior of $K$.
We can see that $K$ is homeomorphic to the Klein bottle with one boundary component.
$Y_{m,a}$ is defined as the isotopy class of a diffeomorphism over $N$ which is described by pushing $M$ once along $a$ and which fixes each point of the boundary and the exterior of $K$ (see Figure~\ref{y-homeo}).

\begin{figure}[htbp]
\includegraphics[scale=0.5]{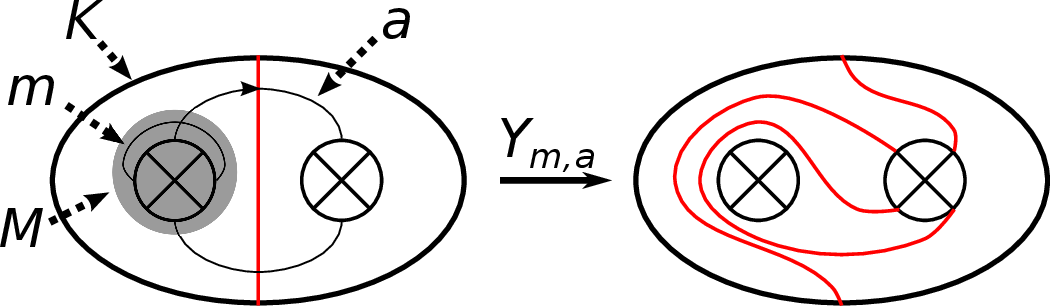}
\caption{The $Y$-homoemorphism $Y_{m,a}$.}\label{y-homeo}
\end{figure}

\subsection{On Torelli groups for non-orientable surfaces}\

Let $c$ be an $A$-circle on a non-orientable surface $N$ such that $N\setminus{c}$ is not connected.
We call $t_c$ a \textit{bounding simple closed curve map}, for short a \textit{BSCC map}.
For example, in Theorem~\ref{main-thm}, $t_\alpha$, $t_\gamma$, $t_{\delta_i}$, $t_{\rho_i}$, $t_{\sigma_{ij}}$ and $t_{\bar{\sigma}_{ij}}$ are BSCC maps.
Let $c_1$ and $c_2$ be $A$-circles on $N$ such that $N\setminus{c_i}$ is connected, $N\setminus(c_1\cup{}c_2)$ is not connected and one of its connected components is an orientable surface with two boundary components.
We call $t_{c_1}t_{c_2}^{-1}$ a \textit{bounding pair map}, for short a \textit{BP map}, where $t_{c_1}$ and $t_{c_2}$ are right-handed twists for some orientation of the orientable component of $N\setminus(c_1\cup{}c_2)$.
For example, in Theorem~\ref{main-thm}, $t_\beta{}t_{\beta^\prime}^{-1}$ is a BP map.

Hirose and the author \cite{HK} obtained the following theorem.
\begin{thm}[\cite{HK}]\label{pre-thm}
For $g\geq5$, $\I(N_g)$ is normally generated by $t_\alpha$ and $t_\beta{}t_{\beta^\prime}^{-1}$ in $\M(N_g)$.
$\I(N_4)$ is normally generated by $t_\alpha$, $t_\beta{}t_{\beta^\prime}^{-1}$ and $t_\gamma$ in $\M(N_4)$.
\end{thm}
Theorem~\ref{main-thm} is a natural extension of Theorem~\ref{pre-thm}.

We can check that all BSCC maps and BP maps are in $\I(N_g^b)$.
In addition, by Theorem~\ref{main-thm}, we have that $\I(N_g^b)$ is generated by BSCC maps and BP maps.
We do not know whether or not $\I(N_g^b)$ can be finitely generated.

\subsection{ Capping, Pushing and Forgetful homomorphisms}\

Let $N$ be a compact non-orientable surface.
Take a point $\ast$ in the interior of $N$.
Let $\M(N,\ast)$ denote the group consisting of isotopy classes of all diffeomorphisms over $N$  which fix $\ast$ and each point of the boundary, and let $\I(N,\ast)$ denote the subgroup of $\M(N,\ast)$ consisting of elements acting trivially on $H_1(N;\Z)$

Take a point $\ast$ in the interior of $N_g^{b-1}$.
We regard $N_g^b$ as a subsurface of $N_g^{b-1}$ not containing $\ast$.
The natural embedding $N_g^b\hookrightarrow{}N_g^{b-1}$ induces the homomorphism $\C_g^b:\M(N_g^b)\to\M(N_g^{b-1},\ast)$, called the \textit{capping homomorphism}.
We have the following lemma.

\begin{lem}[cf. \cite{FM,St2}]\label{kernel-c}
$\ker\C_g^b$ is generated by $t_{\delta_b}$.
\end{lem}

Since $t_{\delta_b}$ is in $\I(N_g^b)$ we obtain the following.

\begin{cor}\label{kernel-c-res}
$\ker\C_g^b|_{\I(N_g^b)}$ is generated by $t_{\delta_b}$.
\end{cor}

We remark that $\C_g^b$ and $\C_g^b|_{\I(N_g^b)}$ are not surjective.

The \textit{pushing homomorphism} $\Pu_g^{b-1}:\pi_1(N_g^{b-1},\ast)\to\M(N_g^{b-1},\ast)$ is defined as follows.
For $x\in\pi_1(N_g^{b-1},\ast)$ take a representative oriented loop $\tilde{x}$ based at $\ast$.
$\Pu_g^{b-1}(x)$ is described by pushing $\ast$ once along $\tilde{x}$ and fixes each point of exterior of neighborhood of $\tilde{x}$.
Note that $\Pu_g^{b-1}$ is an anti-homomorphism, that is, for $x,y\in\pi_1(N_g^{b-1},\ast)$ we have $\Pu_g^{b-1}(xy)=\Pu_g^{b-1}(y)\Pu_g^{b-1}(x)$.
The \textit{forgetful homomorphism} $\F_g^{b-1}:\M(N_g^{b-1},\ast)\to\M(N_g^{b-1})$ is defined naturally.
Note that $\F_g^{b-1}$ is surjective.

We have the exact sequence
$$\pi_1(N_g^{b-1},\ast)\overset{\Pu_g^{b-1}}{\longrightarrow}\M(N_g^{b-1},\ast)\overset{\F_g^{b-1}}{\longrightarrow}\M(N_g^{b-1})\longrightarrow1.$$
This sequence is called the \textit{Birman exact sequence}, introduced by Birman~\cite{B}.
The pushing homomorphism is injective if the Euler characteristic of $N_g^{b-1}$ is negative.
Since $\textrm{Im}\Pu_g^{b-1}$ is in $\I(N_g^{b-1},\ast)$ and $\F_g^{b-1}(\I(N_g^{b-1},\ast))$ is equal to $\I(N_g^{b-1})$, we have the exact sequence
$$\pi_1(N_g^{b-1},\ast)\longrightarrow\I(N_g^{b-1},\ast)\longrightarrow\I(N_g^{b-1})\longrightarrow1.$$

\section{A normal generating set for $\C_g^b(\I(N_g^b))$ in $\C_g^b(\M(N_g^b))$}

For a map $f:X\to{}Y$ and $y\in{}Y$,  we call $x\in{}X$ a \textit{lift} by $f$ of $y$ if $x$ is in $f^{-1}(y)$.

By the capping homomorphism $\C_g^b:\M(N_g^b)\to\M(N_g^{b-1},\ast)$ we have that a normal generating set for $\I(N_g^b)$ in $\M(N_g^b)$ consists of $t_{\delta_b}$ and lifts by $\C_g^b$ of normal generators of $\C_g^b(\I(N_g^b))$  in $\C_g^b(\M(N_g^b))$.
Thus in this section we consider a normal generating set for $\C_g^b(\I(N_g^b))$  in $\C_g^b(\M(N_g^b))$.

Let $\alpha_i$ and $\beta_j$ be oriented loops on $N_g^{b-1}$ based at $\ast$ as shown in Figure~\ref{xy}, and let $x_i$ and $y_j$ be the elements of $\pi_1(N_g^{b-1},\ast)$ corresponding to $\alpha_i$ and $\beta_j$ respectively.
Note that $\pi_1(N_g^{b-1},\ast)$ is generated by $x_i$ and $y_j$ for $1\leq{i}\leq{g}$ and $1\leq{j}\leq{b-1}$.

\begin{figure}[htbp]
\includegraphics[scale=0.5]{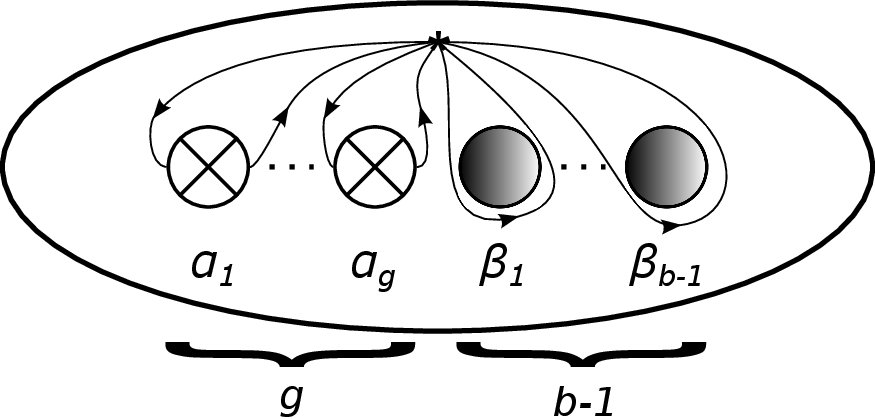}
\caption{The oriented loops $\alpha_i$ and $\beta_j$ on $N_g^{b-1}$.}\label{xy}
\end{figure}

Let $\gamma_i$ and $\delta_j$ be oriented loops on $N_g^b$ as shown in Figure~\ref{cd}, and let $c_i$ and $d_i$ be the elements of $H_1(N_g^b;Z)$ corresponding to $\gamma_i$ and $\delta_j$ respectively.
As a $\Z$-module, $H_1(N_g^b,\Z)$ has the presentation
$$H_1(N_g^b,\Z)=\langle{c_1,\dots,c_g,d_1,\dots,d_b\mid{}2(c_1+\cdots+c_g)+(d_1+\dots+d_b)=0}\rangle,$$
and is isomorphic to $\Z^{g+b-1}$.

\begin{figure}[htbp]
\includegraphics[scale=0.5]{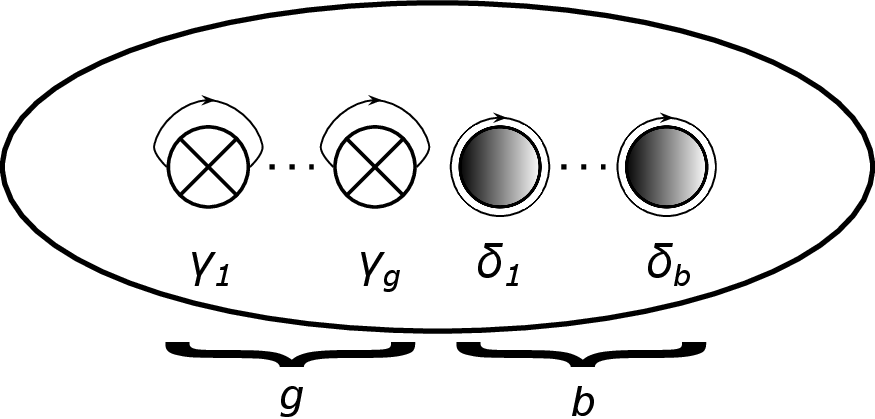}
\caption{The oriented loops $\gamma_i$ and $\delta_j$ on $N_g^b$.}\label{cd}
\end{figure}

Let $p:\pi_1(N_g^{b-1},\ast)\to\pi_1(N_g,\ast)$ denote the natural surjection defined as $p(x_i)=x_i$ and $p(y_j)=1$, for $b\geq1$.
Note that $\Pu_g^{b-1}(x)$ can be lifted by $\C_g^b$ if and only if the word length of $p(x)\in\pi_1(N_g,\ast)$ is even.
Let $\pi_1^+(N_g^{b-1},\ast)$ denote the subgroup of $\pi_1(N_g^{b-1},\ast)$ consisting of $x$ such that the word length of $p(x)$ is even.
For $x\in\pi_1^+(N_g^{b-1},\ast)$, let $x_\ast$ denote the automorphism over $H_1(N_g^b,\Z)$ induced by any lift by $\C_g^b$ of $\Pu_g^{b-1}(x)$.
We can denote $x_\ast(c_i)=c_i+n_id_b$ for some integer $n_i$ and $1\leq{i}\leq{g}$.
Then we define a homomorphism $\theta_g^{b-1}:\pi_1^+(N_g^{b-1},\ast)\to\mathbb{Z}^g$ as $\theta_g^{b-1}(x)=(n_1,n_2,\dots,n_g)$.
Clearly this is a homomorphism.
Let $\Gamma_g^{b-1}=\ker\theta_g^{b-1}$.

In this section we prove the following three propositions.

\begin{prop}\label{im-cap-torelli-n-c-1}
$\mathrm{Im}~\F_g^{b-1}|_{\C_g^b(\I(N_g^b))}$ is equal to $\I(N_g^{b-1})$.
\end{prop}

\begin{prop}\label{im-cap-torelli-n-c-2}
$\ker\F_g^{b-1}|_{\C_g^b(\I(N_g^b))}$ is equal to $\Pu_g^{b-1}(\Gamma_g^{b-1})$.
\end{prop}

\begin{prop}\label{p-g}
$\Pu_g^{b-1}(\Gamma_g^{b-1})$ is the normal closure of $\Pu_g^{b-1}(x_g^2)$, $\Pu_g^{b-1}(y_j)$ and $\Pu_g^{b-1}(x_gy_jx_g^{-1})$ for $1\leq{j}\leq{b-1}$ in $\C_g^b(\M(N_g^b))$.
\end{prop}

By Proposition~\ref{im-cap-torelli-n-c-1} and Proposition~\ref{im-cap-torelli-n-c-2}, we have the exact sequence
$$\Gamma_g^{b-1}\longrightarrow\C_g^b(\I(N_g^b))\longrightarrow\I(N_g^{b-1})\longrightarrow1.$$
Hence $\C_g^b(\I(N_g^b))$ is the normal closure of $\Pu_g^{b-1}(\Gamma_g^{b-1})$ and lifts by $\F_g^{b-1}|_{\C_g^b(\I(N_g^b))}$ of normal generators of $\I(N_g^{b-1})$.
In addition, by Proposition~\ref{p-g} we obtain the following.

\begin{cor}\label{im-cap-torelli-n-g}
In $\C_g^b(\M(N_g^b))$, $\C_g^b(\I(N_g^b))$ is normally generated by $\Pu_g^{b-1}(x_g^2)$, $\Pu_g^{b-1}(y_j)$, $\Pu_g^{b-1}(x_gy_jx_g^{-1})$ and lifts by $\F_g^{b-1}|_{\C_g^b(\I(N_g^b))}$ of normal generators of $\I(N_g^{b-1})$, for $1\leq{j}\leq{b-1}$.
\end{cor}

\subsection{Proof of Proposition~\ref{im-cap-torelli-n-c-1}}\

To prove Proposition~\ref{im-cap-torelli-n-c-1}, it suffices to show that for any $\varphi\in\I(N_g^{b-1})$ there is $\widetilde{\varphi}\in\I(N_g^b)$ such that $(\F_g^{b-1}\circ\C_g^b)(\widetilde{\varphi})=\varphi$.

For $f\in\M(N_g^b)$ we denote by $f_\ast$ the automorphism over $H_1(N_g^b;\Z)$ induced by $f$.
Since $f$ fixes each point of the boundary of $N_g^b$, we have that $f_\ast(d_j)=d_j$ for $1\leq{j}\leq{b}$.

For any $\varphi\in\I(N_g^{b-1})$ there exists $\psi\in\M(N_g^b)$ such that $(\F_g^{b-1}\circ\C_g^b)(\psi)=\varphi$.
For $1\leq{i}\leq{g}$ there is an integer $n_i$ such that $\psi_\ast(c_i)=c_i+n_id_b$.
Let $\gamma_{ij}$ and $\widetilde{\gamma}_{ij}$ be simple closed curves on $N_g^b$ as shown in Figure~\ref{g_ij}, and let $\tau_{ij}=t_{\gamma_{ij}}t_{\widetilde{\gamma}_{ij}}$, for $1\leq{i<j}\leq{g}$.
$\tau_{ij}$ is a mapping class which is described by pushing the $b$-th boundary component once along between $\gamma_{ij}$ and $\widetilde{\gamma}_{ij}$.
We can check
$$
(\tau_{ij})_\ast(c_t)=
\left\{
\begin{array}{ll}
c_i-d_b&(t=i),\\
c_j+d_b&(t=j),\\
c_t&(t\neq{i,j}).
\end{array}
\right.
$$

\begin{figure}[htbp]
\includegraphics[scale=0.5]{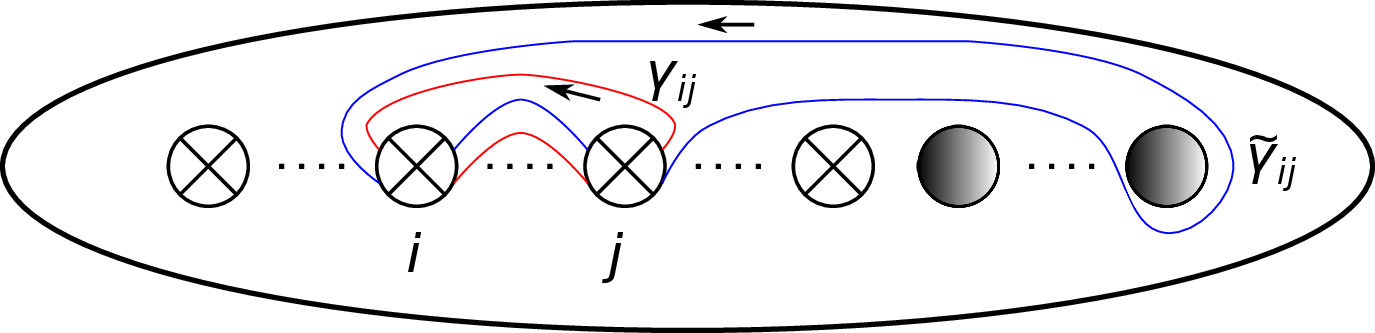}
\caption{The loops $\gamma_{ij}$ and $\widetilde{\gamma}_{ij}$.}\label{g_ij}
\end{figure}

Let $\tau=\tau_{g-1g}^{n_{g-1}}\cdots\tau_{2g}^{n_2}\tau_{1g}^{n_1}$ and $\widetilde{\varphi}=\tau\psi$.
Since $(\F_g^{b-1}\circ\C_g^b)(\tau)=1$ we have $(\F_g^{b-1}\circ\C_g^b)(\widetilde{\varphi})=\varphi$.
For $1\leq{i}\leq{g-1}$ we calculate
\begin{eqnarray*}
\widetilde{\varphi}_\ast(c_i)
&=&\tau_\ast(\psi_\ast(c_i))\\
&=&\tau_\ast(c_i+n_id_b)\\
&=&\tau_\ast(c_i)+n_i\tau_\ast(d_b)\\
&=&(\tau_{ig})^{n_i}_\ast(c_i)+n_id_b\\
&=&(c_i-n_id_b)+n_id_b\\
&=&c_i.
\end{eqnarray*}
In addition, we calculate
\begin{eqnarray*}
\widetilde{\varphi}_\ast(c_g)
&=&\tau_\ast(\psi_\ast(c_g))\\
&=&\tau_\ast(c_g+n_gd_b)\\
&=&\tau_\ast(c_g)+n_g\tau_\ast(d_b)\\
&=&(c_g+\sum_{k=1}^{g-1}n_kd_b)+n_gd_b\\
&=&c_g+\sum_{k=1}^gn_kd_b.
\end{eqnarray*}
Let $c=2(c_1+\cdots+c_g)+(d_1+\dots+d_b)(=0)$.
We see
\begin{eqnarray*}
\widetilde{\varphi}_\ast(c)
&=&2(\widetilde{\varphi}_\ast(c_1)+\cdots+\widetilde{\varphi}_\ast(c_g))+(\widetilde{\varphi}_\ast(d_1)+\cdots+\widetilde{\varphi}_\ast(d_b))\\
&=&2(c_1+\cdots+c_{g-1}+c_g+\sum_{k=1}^gn_kd_b)+(d_1+\dots+d_b)\\
&=&c+2\sum_{k=1}^gn_kd_b\\
&=&2\sum_{k=1}^gn_kd_b.
\end{eqnarray*}
Since $\widetilde{\varphi}_\ast(c)=0$ we have $\displaystyle\sum_{k=1}^gn_k=0$, and hence $\widetilde{\varphi}_\ast(c_g)=c_g$.
Hence we have that $\widetilde{\varphi}_\ast$ is the identity.
Therefore we conclude that $\widetilde{\varphi}$ is in $\I(N_g^b)$ with $(\F_g^{b-1}\circ\C_g^b)(\widetilde{\varphi})=\varphi$.
Thus we complete the proof of Proposition~\ref{im-cap-torelli-n-c-1}.

\subsection{Proof of Proposition~\ref{im-cap-torelli-n-c-2}}\

Note that $\ker\F_g^{b-1}|_{\C_g^b(\I(N_g^b))}$ is equal to the intersection of $\ker\F_g^{b-1}$ and $\C_g^b(\I(N_g^b))$.
For any $x\in\Gamma_g^{b-1}$, since $x_\ast=1$, we have that $\Pu_g^{b-1}(x)$ is in $\C_g^b(\I(N_g^b))$.
In addition, since $\Pu_g^{b-1}(x)$ is in $\ker\F_g^{b-1}$, we have that $\Pu_g^{b-1}(x)$ is in $\ker\F_g^{b-1}|_{\C_g^b(\I(N_g^b))}$.
Hence we conclude $\Pu_g^{b-1}(\Gamma_g^{b-1})\subset\ker\F_g^{b-1}|_{\C_g^b(\I(N_g^b))}$.

For any $\varphi\in\ker\F_g^{b-1}|_{\C_g^b(\I(N_g^b))}$, since $\varphi$ is in $\ker\F_g^{b-1}$ there exists $x\in\pi_1(N_g^{b-1},\ast)$ such that $\varphi=\Pu_g^{b-1}(x)$.
In addition, since $\varphi$ is in $\C_g^b(\I(N_g^b))$ we have $x_\ast=1$, and hence $x\in\Gamma_g^{b-1}$.
Hence $\varphi$ is in $\Pu_g^{b-1}(\Gamma_g^{b-1})$.
Therefore we conclude $\ker\F_g^{b-1}|_{\C_g^b(\I(N_g^b))}\subset\Pu_g^{b-1}(\Gamma_g^{b-1})$.

Thus we complete the proof of Proposition~\ref{im-cap-torelli-n-c-2}.

\subsection{Proof of Proposition~\ref{p-g}}\

We have the exact sequence
$$\pi_1^+(N_g^{b-1},\ast)\longrightarrow\C_g^b(\M(N_g^b))\longrightarrow\M(N_g^{b-1})\longrightarrow1.$$
Hence a normal generator of $\Pu_g^{b-1}(\Gamma_g^{b-1})$ in $\C_g^b(\M(N_g^b))$ is the image of a normal generator of $\Gamma_g^{b-1}$ in $\pi_1^+(N_g^{b-1},\ast)$.
Therefore we consider about the normal generators of $\Gamma_g^{b-1}$ in $\pi_1^+(N_g^{b-1},\ast)$.

We show the following three lemmas.

\begin{lem}\label{pi+-lem1}
For $b\geq1$, $\pi_1^+(N_g^{b-1},\ast)$ is generated by $x_ix_g^{-1}$, $x_gx_j$, $y_k$ and $x_gy_kx_g^{-1}$, where $1\leq{i}\leq{g-1}$, $1\leq{j}\leq{g}$ and $1\leq{k}\leq{b-1}$.
\end{lem}

\begin{proof}
We use the Reidemeister Schreier method.
$\pi_1^+(N_g^{b-1},\ast)$ is an index $2$ subgroup of $\pi_1(N_g^{b-1},\ast)$.
Let $U=\{1,x_g\}$.
Remark that $U$ is a Schreier transversal for $\pi_1^+(N_g^{b-1},\ast)$ in $\pi_1(N_g^{b-1},\ast)$.
Let $X=\{x_1,\dots,x_g,y_1,\dots,y_{b-1}\}$.
For $u\in{U}$ and $x\in{X}$, $\overline{ux}=1$ if $(u,x)=(1,y_j)$ or $(x_g,x_i)$ (resp. $\overline{ux}=x_g$ if $(u,x)=(1,x_i)$ or $(x_g,y_j)$).
A generating set of $\pi_1^+(N_g^{b-1},\ast)$ is defined as $B=\{ux\overline{ux}^{-1}\mid{}u\in{U},x\in{X},ux\notin{U}\}$.
We see $1x_i\overline{1x_i}^{-1}=x_ix_g^{-1}$, $x_gx_j\overline{x_gx_j}^{-1}=x_gx_j$, $1y_k\overline{1y_k}^{-1}=y_k$, $x_gy_k\overline{x_gy_k}^{-1}=x_gy_kx_g^{-1}$, for $1\leq{i}\leq{g-1}$, $1\leq{j}\leq{g}$ and $1\leq{k}\leq{b-1}$.
Since $1x_g\overline{1x_g}^{-1}=x_gx_g^{-1}=1\in{U}$, $1x_g\overline{1x_g}^{-1}$ is not in $B$.
Thus we obtain the claim.
\end{proof}

\begin{lem}\label{pi+-lem2}
For $b\geq1$, $\Gamma_g^{b-1}$ is normally generated by $x_ix_g^{-1}\cdot{}x_gx_i$, $x_gx_g$, $y_j$, $x_gy_jx_g^{-1}$ and the commutators $[x_gx_k,x_gx_l]$ in $\pi_1^+(N_g^{b-1},\ast)$, where $1\leq{i}\leq{g-1}$, $1\leq{j}\leq{b-1}$ and $1\leq{k<l}\leq{g-1}$.
\end{lem}

\begin{proof}
Recall the homomorphism $\theta_g^{b-1}:\pi_1^+(N_g^{b-1},\ast)\to\Z^g$.
By the proof of Proposition~\ref{im-cap-torelli-n-c-1} we have
$$\mathrm{Im}~\theta_g^{b-1}=\{(n_1,\dots,n_g)\in\mathbb{Z}^g\mid{}n_1+\cdots+n_g=0\}\cong\mathbb{Z}^{g-1}.$$
$\mathrm{Im}~\theta_g^{b-1}$ is generated by the vectors $e_k$ with $1$ on position $k$, $-1$ on position $g$, and $0$ on elsewhere, for $1\leq{k}\leq{}g-1$.
Let $H$ be the normal closure in $\pi_1^+(N_g^{b-1},\ast)$ of the set of elements given in the lemma and $G$ be the quotient of $\pi_1^+(N_g^{b-1},\ast)$ by $H$.
Since $H$ is the subgroup of $\Gamma_g^{b-1}=\ker\theta_g^{b-1}$, there is an induced epimorphism $\theta^\prime:G\to\mathrm{Im}~\theta_g^{b-1}$.
By Lemma~\ref{pi+-lem1}, $G$ is an abelian group generated by $x_gx_k$ for $1\leq{k}\leq{}g-1$.
Note that $\theta^\prime(x_gx_k)=e_k$.
Since $\mathrm{Im}~\theta_g^{b-1}$ is a free abelian group of rank $g-1$, $\theta^\prime$ is an isomorphism.
Therefor we have $H=\Gamma_g^{b-1}$.
Thus we obtain the claim.
\end{proof}

\begin{lem}\label{pi+-lem3}
For $b\geq1$, $\Gamma_g^{b-1}$ is normally generated by $x_ix_g^{-1}\cdot{}x_gx_i$, $x_gx_g$, $y_j$, $x_gy_jx_g^{-1}$ and $(x_kx_lx_g)^2$ in $\pi_1^+(N_g^{b-1},\ast)$, where $1\leq{i}\leq{g-1}$, $1\leq{j}\leq{b-1}$ and $1\leq{k<l}\leq{g-1}$.
\end{lem}

\begin{proof}
For any $1\leq{k<l}\leq{g-1}$, modulo $x_1^2,x_2^2,\dots,x_g^2$ we see
\begin{eqnarray*}
[x_gx_k,x_gx_l]^{-1}
&=&
[x_gx_l,x_gx_k]\\
&=&
x_gx_lx_gx_kx_l^{-1}x_g^{-1}x_k^{-1}x_g^{-1}\\
&=&
x_gx_lx_gx_kx_l^{-1}x_g^{-1}x_k^{-1}x_lx_gx_g^{-1}x_l^{-1}x_g^{-1}\\
&\equiv&
x_gx_lx_g(x_kx_lx_g)^2x_g^{-1}x_l^{-1}x_g^{-1}
\end{eqnarray*}
Hence $(x_kx_lx_g)^2$ is conjugate to $[x_gx_k,x_gx_l]^{-1}$ modulo $x_1^2,x_2^2,\dots,x_g^2$.
By Lemma~\ref{pi+-lem2} we obtain the claim.
\end{proof}

By Lemma~\ref{pi+-lem3}, $\Pu_g^{b-1}(\Gamma_g^{b-1})$ is normally generated by $\Pu_g^{b-1}(x_i^2)$, $\Pu_g^{b-1}(y_j)$, $\Pu_g^{b-1}(x_gy_jx_g^{-1})$ and $\Pu_g^{b-1}((x_kx_lx_g)^2)$ for $1\leq{i}\leq{g-1}$, $1\leq{j}\leq{b-1}$ and $1\leq{k<l}\leq{g-1}$.
Representative loops of $x_i^2$ and $(x_kx_lx_g)^2$ bound a M\"obius band (see Figure~\ref{mobius}).
Therefore $\Pu_g^{b-1}(x_i^2)$ and $\Pu_g^{b-1}((x_kx_lx_g)^2)$ are conjugate to $\Pu_g^{b-1}(x_g^2)$.
Thus we complete the proof of Proposition~\ref{p-g}.

\begin{figure}[htbp]
\subfigure[Representative loops of $x_1^2,x_2^2,\dots,x_g^2$ bound a M\"obius band.]{\includegraphics[scale=0.5]{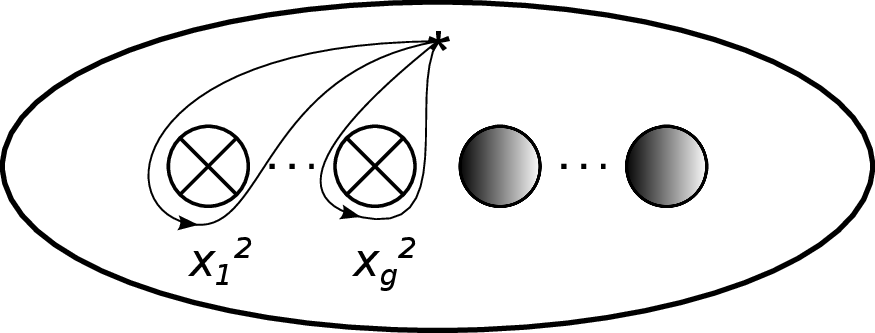}}
\subfigure[A representative loop of $(x_1x_2x_3)^2$ bounds a M\"obius band. Similarly a representative loop of $(x_kx_lx_g)^2$ bounds a M\"obius band for $1\leq{k<l}\leq{g-1}$.]{\includegraphics[scale=0.5]{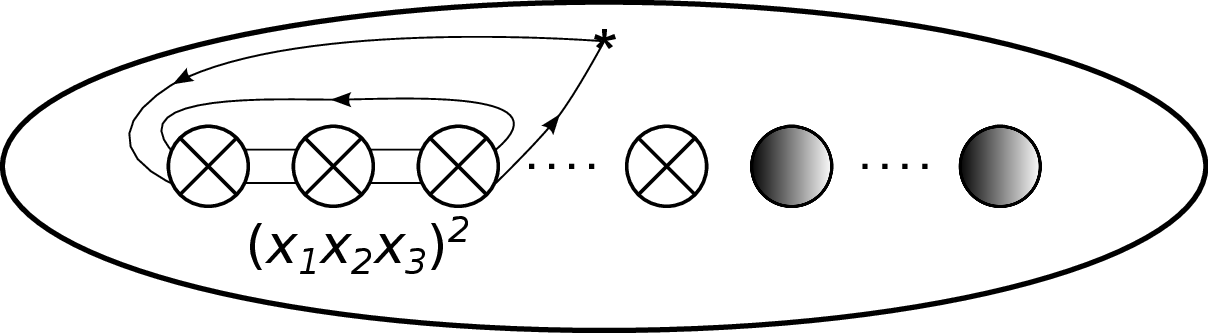}}
\caption{}\label{mobius}
\end{figure}

\section{Proof of Theorem~\ref{main-thm}}

We have the following short exact sequence:
$$1\to\ker\C_g^b|_{\I(N_g^b)}\to\I(N_g^b)\to\C_g^b(\I(N_g^b))\to1.$$
Hence $\I(N_g^b)$ is the normal closure of $\ker\C_g^b|_{\I(N_g^b)}$ and lifts by $\C_g^b$ of normal generators of $\C_g^b(\I(N_g^b))$.
By Corollary~\ref{kernel-c-res} and Corollary~\ref{im-cap-torelli-n-g}, we have that $\I(N_g^b)$ is normally generated by $t_{\delta_b}$, the lift $t_{\rho_b}$ by $\C_g^b$ of $\Pu_g^{b-1}(x_g^2)$, the lift $t_{\sigma_{jb}}t_{\delta_j}^{-1}$ by $\C_g^b$ of $\Pu_g^{b-1}(y_j)$, the lift $t_{\bar{\sigma}_{jb}}t_{\delta_j}^{-1}$ by $\C_g^b$ of $\Pu_g^{b-1}(x_gy_jx_g^{-1})$ and lifts by $(\F_g^{b-1}\circ\C_g^b)|_{\I(N_g^b)}$ of normal generators of $\I(N_g^{b-1})$.
We prove Theorem~\ref{main-thm} by induction on the number $b$ of the boundary components of $N_g^b$.
We take the natural lift by $(\F_g^{b-1}\circ\C_g^b)|_{\I(N_g^b)}$ of a normal generator of $\I(N_g^{b-1})$.

At first, $\I(N_g)$ is normally generated by $t_\alpha$, $t_\beta{}t_{\beta^\prime}^{-1}$ and $t_\gamma$ (see \cite{HK}).
Hence we have that $\I(N_g^1)$ is normally generated by $t_\alpha$, $t_\beta{}t_{\beta^\prime}^{-1}$, $t_\gamma$ and $t_{\delta_1}$, $t_{\rho_1}$.
Similarly we have that $\I(N_g^2)$ is normally generated by $t_\alpha$, $t_\beta{}t_{\beta^\prime}^{-1}$, $t_\gamma$, $t_{\delta_1}$, $t_{\rho_1}$ and $t_{\delta_2}$, $t_{\rho_2}$, $t_{\sigma_{12}}$, $t_{\bar{\sigma}_{12}}$.
For $b\geq3$, suppose that $\I(N_g^{b-1})$ is normally generated by $t_\alpha$, $t_\beta{}t_{\beta^\prime}^{-1}$, $t_\gamma$, $t_{\delta_i}$, $t_{\rho_i}$, $t_{\sigma_{ij}}$ and $t_{\bar{\sigma}_{ij}}$ for $1\leq{i,j}\leq{b-1}$ with $i<j$.
Then we have that $\I(N_g^b)$ is normally generated by $t_\alpha$, $t_\beta{}t_{\beta^\prime}^{-1}$, $t_\gamma$, $t_{\delta_i}$, $t_{\rho_i}$, $t_{\sigma_{ij}}$, $t_{\bar{\sigma}_{ij}}$ and $t_{\delta_b}$, $t_{\rho_b}$, $t_{\sigma_{kb}}$, $t_{\bar{\sigma}_{kb}}$ for $1\leq{i,j,k}\leq{b-1}$ with $i<j$.
Hence we obtain a normal generating set for $\I(N_g^b)$.
In particular, for $g\geq5$ since $\I(N_g)$ is normally generated by $t_\alpha$ and $t_\beta{}t_{\beta^\prime}^{-1}$ (see \cite{HK}), we do not need $t_\gamma$ as a normal generator of $\I(N_g^b)$ for $g\geq5$.
Finally, proving the following lemma, we finish the proof of Theorem~\ref{main-thm}.

\begin{lem}
$t_{\delta_b}$, $t_{\rho_b}$, $t_{\sigma_{kb}}$ and $t_{\bar{\sigma}_{kb}}$ are not needed as normal generators of $\I(N_g^b)$.
\end{lem}

\begin{proof}
Let $a_{ij}$, $b_{jk}$ and $c_{kl}$ be simple closed curves on $N_g^b$ as shown in Figure~\ref{abc}, for $1\leq{i,j}\leq{g}$ and $1\leq{k,l}\leq{b-1}$.
Let $d_m$ be a diffeomorphism defined by pushing the $m$-th boundary component once along the loop as shown in Figure~\ref{d}.
Remark that the isotopy class of $d_m$ is not in $\M(N_g^b)$, since $d_m$ does not fix boundary.
$t_{a_{ij}}$ and $t_{d_m(a_{ij})}$ are conjugate to $t_\alpha$.
$t_{b_{jk}}$ and $t_{d_m(b_{jk})}$ are conjugate to ether $t_{\rho_k}$ or $t_{\rho_k}^{-1}$.
$t_{c_{kl}}$ and $t_{d_m(c_{kl})}$ are conjugate to $t_{\sigma_{kl}}$, with $m\neq{k,l}$.
$t_{d_k(c_{kl})}$ and $t_{d_l(c_{kl})}$ are conjugate to $t_{\bar{\sigma}_{kl}}^{-1}$ and $t_{\bar{\sigma}_{kl}}$, respectively.
Hence it suffices to show that $t_{\delta_b}$, $t_{\rho_b}$, $t_{\sigma_{kb}}$ and $t_{\bar{\sigma}_{kb}}$ are products of these Dehn twists.

\begin{figure}[htbp]
\includegraphics[scale=0.5]{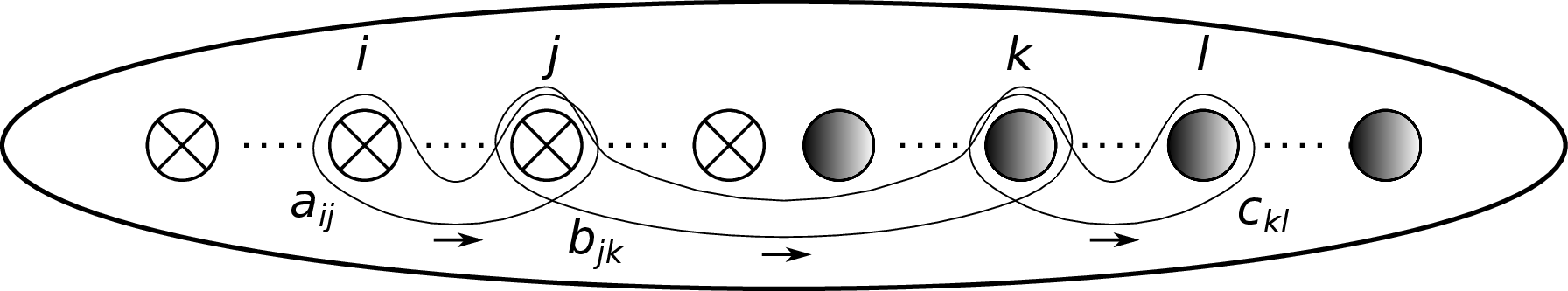}
\caption{The loops $a_{ij}$, $b_{jk}$ and $c_{kl}$.}\label{abc}
\end{figure}

\begin{figure}[htbp]
\includegraphics[scale=0.5]{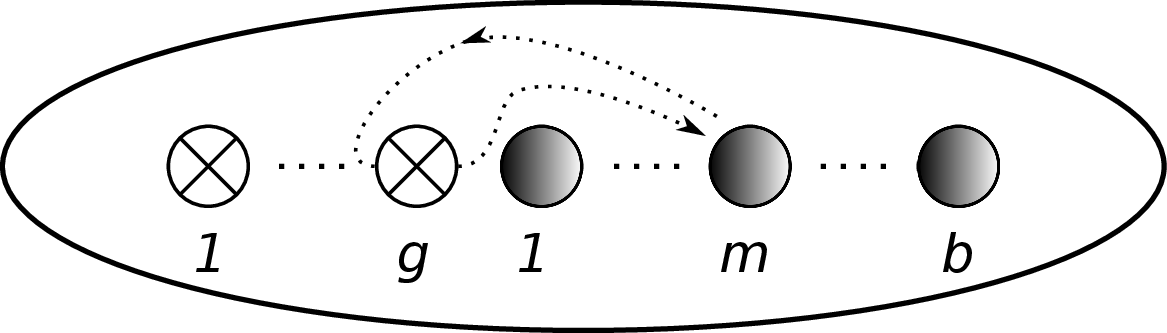}
\caption{The diffeomorphism $d_m$.}\label{d}
\end{figure}

For simplicity, we denote $t_{a_{ij}}=a_{i,j}$, $t_{b_{jk}}=b_{j,k}$, $t_{c_{kl}}=c_{k,l}$, $t_{d_m(a_{ij})}=a_{i,j;m}$, $t_{d_m(b_{jk})}=b_{j,k;m}$ and $t_{d_m(c_{kl})}=c_{k,l;m}$.
We remark that $a_{i,j}$, $b_{j,k}$ and $c_{k,l}$ are described as shown in Figure~\ref{abc-2}.

\begin{figure}[htbp]
\includegraphics[scale=0.5]{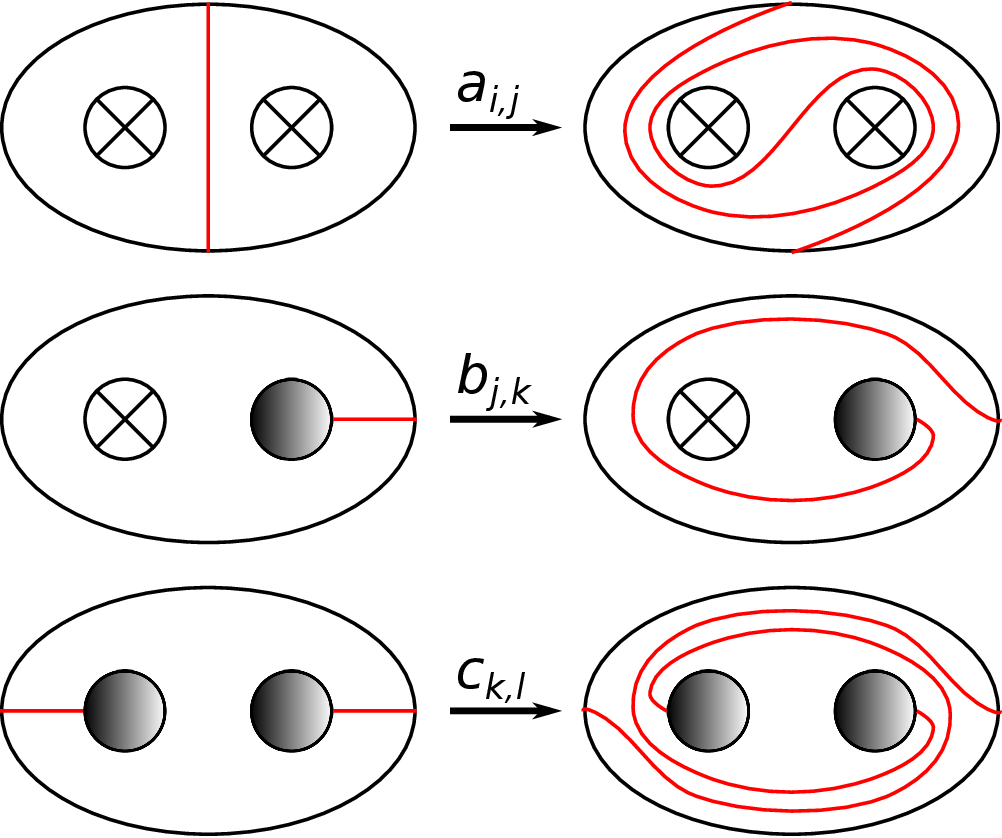}
\caption{}\label{abc-2}
\end{figure}

For $1\leq{i}\leq{g+b-1}$ let $D_i$ be a simple closed curve on $N_g^b$ as shown in Figure~\ref{D_i}.
Then we have $t_{\delta_b}=(t_{D_1}t_{D_2}^{-1})\cdots{}(t_{D_{g+b-2}}t_{D_{g+b-1}}^{-1})t_{D_{g+b-1}}$.
Note that $t_{D_1}=t_{\delta_b}$ and $t_{D_{g+b-1}}=t_{\delta_{b-1}}$.

\begin{figure}[htbp]
\includegraphics[scale=0.5]{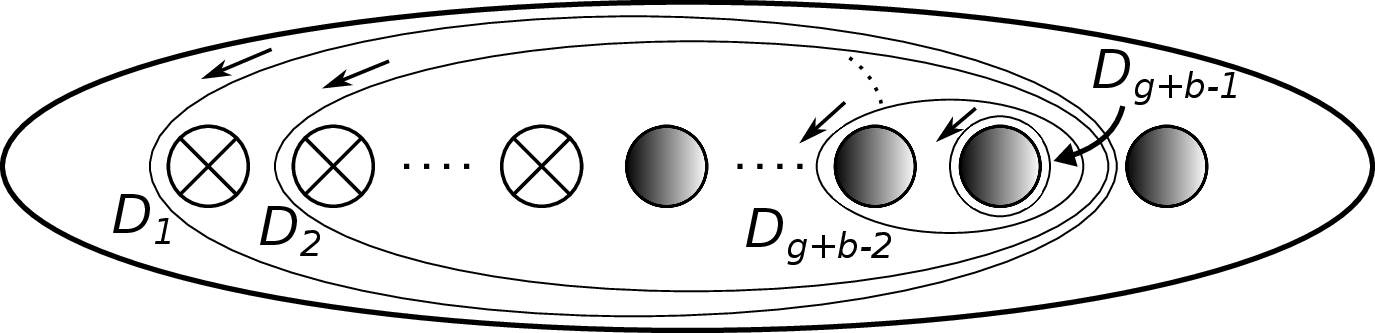}
\caption{The simple closed curves $D_1,D_2,\dots,D_{g+b-1}$.}\label{D_i}
\end{figure}

For $1\leq{i}\leq{g}$, $t_{D_i}t_{D_{i+1}}^{-1}$ is described by pushing the $i$-th crosscap once along between $D_i$ and $D_{i+1}$.
For $1\leq{j}\leq{b-2}$, $t_{D_{g+j}}t_{D_{g+j+1}}^{-1}$ is described by pushing the $j$-th boundary component once along between $D_{g+j}$ and $D_{g+j+1}$.
Similarly, $a_{i,j}$, $b_{j,k}t_{\delta_k}^{-1}$ and $c_{k,l}t_{\delta_{l}^{-1}}$ are described by pushing the $i$-th crosscap, $j$-th crosscap and the $k$-th boundary component around the $j$-th crosscap, the $k$-th boundary component and the $l$-th boundary component, respectively (see Figure~\ref{abc-2}).
Hence for $1\leq{i}\leq{g}$ and $1\leq{j}\leq{b-2}$, we have that $t_{D_i}t_{D_{i+1}}^{-1}$ and $t_{D_{g+j}}t_{D_{g+j+1}}^{-1}$ are explicitly described by products of Dehn twists as
\begin{eqnarray*}
t_{D_i}t_{D_{i+1}}^{-1}&=&a_{i,i+1}\cdots{}a_{i,g}\cdot(b_{i,1}t_{\delta_1}^{-1})\cdots(b_{i,b-1}t_{\delta_{b-1}}^{-1}),\\
t_{D_{g+j}}t_{D_{g+j+1}}^{-1}&=&(c_{j,j+1}t_{\delta_{j+1}}^{-1})\cdots(c_{j,b-1}t_{\delta_{b-1}}^{-1})\cdot{}t_{\delta_j}^{-(b-j-2)}.
\end{eqnarray*}
Therefore we have
\begin{eqnarray*}
t_{\delta_b}
&=&
(t_{\delta_1}\cdots{}t_{\delta_{b-1}})^{-g-b+3}\\
&&
(a_{1,2}\cdots{}a_{1,g}\cdot{}b_{1,1}\cdots{}b_{1,b-1})\cdots
(a_{g-1,g}\cdot{}b_{g-1,1}\cdots{}b_{g-1,b-1})
(b_{g,1}\cdots{}b_{g,b-1})\\
&&
(c_{1,2}\cdots{}c_{1,b-1})\cdots
(c_{b-3,b-2}c_{b-3,b-1})
(c_{b-2,b-1}).
\end{eqnarray*}
Similarly, $t_{\rho_b}$ and $t_{\sigma_{kb}}$ are explicitly described by products of Dehn twists as follows.
\begin{eqnarray*}
t_{\rho_b}
&=&
(t_{\delta_1}\cdots{}t_{\delta_{b-1}})^{-g-b+4}\\
&&
(a_{1,2}\cdots{}a_{1,g-1}\cdot{}b_{1,1}\cdots{}b_{1,b-1})\cdots
(a_{g-2,g-1}\cdot{}b_{g-2,1}\cdots{}b_{g-2,b-1})
(b_{g-1,1}\cdots{}b_{g-1,b-1})\\
&&
(c_{1,2}\cdots{}c_{1,b-1})\cdots
(c_{b-3,b-2}c_{b-3,b-1})
(c_{b-2,b-1}),\\
t_{\sigma_{kb}}
&=&
(t_{\delta_1}\cdots{}t_{\delta_{k-1}}\cdot{}t_{\delta_{k+1}}\cdots{}t_{\delta_{b-1}})^{-g-b+4}\\
&&
(a_{1,2}\cdots{}a_{1,g}\cdot{}b_{1,1}\cdots{}b_{1,k-1}\cdot{}b_{1,k+1}\cdots{}b_{1,b-1})\cdots\\
&&
(a_{g-1,g}\cdot{}b_{g-1,1}\cdots{}b_{g-1,k-1}\cdot{}b_{g-1,k+1}\cdots{}b_{g-1,b-1})
(b_{g,1}\cdots{}b_{g,k-1}\cdot{}b_{g,k+1}\cdots{}b_{g,b-1})\\
&&
(c_{1,2}\cdots{}c_{1,k-1}\cdot{}c_{1,k+1}\cdots{}c_{1,b-1})\cdots
(c_{k-2,k-1}\cdot{}c_{k-2,k+1}\cdots{}c_{k-2,b-1})\\
&&
(c_{k-1,k+1}\cdots{}c_{k-1,b-1})
(c_{k+1,k+2}\cdots{}c_{k+1,b-1})\cdots
(c_{b-3,b-2}c_{b-3,b-1})(c_{b-2,b-1}).
\end{eqnarray*}
In addition, since $t_{\bar{\sigma}_{kb}}=t_{d_k(\sigma_{kb})}^{-1}$, $t_{\bar{\sigma}_{kb}}$ is a product of $a_{i,j;k}$, $b_{i,j;k}$, $c_{i,j;k}$ and $t_{\delta_i}$.

Thus we obtain the claim.
\end{proof}

\section*{Acknowledgement}
The author would like to express his thanks to Genki Omori for his valuable suggestions and useful comments.


\end{document}